\documentclass[a4paper,10pt]{amsart}
\usepackage[plainpages=false]{hyperref}
\usepackage{amsfonts,amssymb,amsthm, mathrsfs, lscape}
\usepackage{verbatim}

\usepackage{latexsym}
\usepackage{amssymb}
\usepackage{graphics}
\usepackage{graphicx}
\usepackage[space]{cite}

\usepackage{latexsym}
\usepackage{amsfonts}
\usepackage{amsmath}
\usepackage{latexsym}
\usepackage{amsfonts}
\usepackage{amssymb}
\usepackage{rawfonts}

 \usepackage[usenames,dvipsnames]{pstricks}
 \usepackage{epsfig}
 \usepackage{pst-grad} 
 \usepackage{pst-plot} 
 \usepackage[space]{grffile} 
 \usepackage{etoolbox} 
 \makeatletter 

\renewcommand{\epsilon}{\varepsilon}

\newcommand{\kahler}{K\"ahler }

\newcommand{\PP}{{\mathbb P}}

\newcommand{\C}{{\mathbb C}}

\newcommand{\Z}{{\mathbb Z}}

\renewcommand{\phi}{\varphi}

\newcommand{\hcal}{\mathcal{H}}

\newtheorem{thm}{Theorem}[section]
\newtheorem{theorem}{{Theorem}}[section]

\newtheorem{cor}[theorem]{{Corollary}}
\newtheorem{lem}[theorem]{{Lemma}}
\newtheorem{prop}[theorem]{{Proposition}}

\newenvironment{rem}{\medskip\noindent{\it Remark:\/} }{\medskip}

\newtheorem{conj}[thm]{Conjecture}

\newtheorem{rmk}[thm]{Remark}

\theoremstyle{definition}

\numberwithin{equation}{section}

\def \C {\mathbb C}
\def \Z {\mathbb Z}

\def \P {\mathbb P}

\def \p {\partial}
\def \bp {\bar{\partial}}

\title[projective embedding of log riemann surfaces and K-stability]{projective embedding of log riemann surfaces and K-stability }

\author{Jingzhou Sun and Song Sun}
\thanks{The second author is partially supported by NSF grant DMS-1405832  and Alfred P. Sloan fellowship.}
\address{Department of Mathematics, Shantou University, Shantou City, Guangdong Province 515063, China}
\address{Department of Mathematics, Stony Brook University, Stony Brook, NY 11794, USA}
\email{jzsun@stu.edu.cn, song.sun@stonybrook.edu}

\begin{document}

\begin{abstract}
Given a smooth polarized Riemann surface $(X, L)$ endowed with a hyperbolic metric $\omega$ that has standard cusp singularities along a divisor $D$, we show the $L^2$ projective embedding of $(X, D)$ defined by $L^k$ is asymptotically almost balanced in a weighted sense. The proof depends on sufficiently precise understanding of the behavior of the Bergman kernel in three regions, with the most crucial one being the neck region around $D$. This is the first step towards understanding the algebro-geometric stability of extremal K\"ahler metrics with singularities. 
\end{abstract}

\maketitle


\section{Introduction}
Let $(X, L)$ be an $n$ dimensional polarized K\"ahler manifold. The famous Yau-Tian-Donaldson conjecture relates the existence of constant scalar curvature K\"ahler (cscK) metrics in the class $2\pi c_1(L)$ to the K-stability of $(X, L)$. This is essentially a correspondence between differential geometry/PDE and algebraic geometry of $(X, L)$.  The direction from cscK metrics to K-stability was established by Donaldson \cite{donaldson2001}, Stoppa \cite{Stoppa}, Mabuchi \cite{Mabuchi}, using the idea of  \emph{quantization}. The other direction is much more involved, and it has been established for toric surfaces by Donaldson \cite{Donaldson2009}, and for anti-canonically polarized Fano manifolds by the recent result of Chen-Donaldson-Sun \cite{CDS1, CDS2, CDS3} (the corresponding metrics are \emph{K\"ahler-Einstein}). 

A crucial ingredient in the proof of \cite{CDS1, CDS2, CDS3} is the introduction of a smooth divisor $D\in |-mK_X|$ for some $m\geq 1$. Both aspects of the above conjecture extend naturally to the pair $(X, D)$ with an extra parameter $\beta\in [0, 1]$. On the algebraic geometric side we have a notion of \emph{logarithmic K-stability} for $(X, D, K_X^{-1}, \beta)$, and on the differential geometric side the corresponding object is a K\"ahler-Einstein metric with cone angle $2\pi\beta$ along $D$. Roughly speaking the strategy of \cite{CDS1, CDS2, CDS3} is a continuous deformation from $\beta=0$ to $\beta=1$. A simple but important fact is that the logarithmic K-stability is linear in $\beta$, and it is then evidently important to study both aspects at $\beta=0$.  On the metric side one expects complete K\"ahler-Einstein metrics on the complement $X\setminus D$, and such metrics are known to exist (\cite{ChengYau, TianYau1, Kobayashi}), by adapting Yau's solution of the Calabi conjecture and following Calabi's ansatz; on the algebraic side the K-semistability of $(X, D, K_X^{-1}, 0)$ is established by \cite{ssun}, \cite{OdakaSun}, \cite{Berman}, \cite{LiSun}. However, a direct relationship between these two facts seems missing.

Now for a general polarized manifold $(X, L)$ with a smooth divisor $D$ the above discussion can be extended in a straightforward way by replacing $K_X^{-1}$ with $L$. Such a theory has not yet been satisfactorily established. In this direction we expect the following 

\begin{conj} \label{conj1-1}
Let $(X, L)$ be a polarized K\"ahler manifold of dimension $n$, and $D$ a smooth divisor in the class $c_1(L)$.  Denote $\sigma=-(K_X+L). L^{n-1}/L^n$, and suppose $D$ admits a constant scalar curvature K\"ahler metric $\omega_D\in 2\pi c_1(L|_D)$. Then  $(X, D, L, 0)$ is logarithmic K-semistable if $\sigma\leq 0$. 
\end{conj} 

Notice the sign of $\sigma$ is the same as the sign of the scalar curvature of $\omega_D$. When $\sigma=0$ Conjecture \ref{conj1-1} follows from \cite{ssun} (the proof there is written assuming $D$ is Calabi-Yau, but it is easy to see one only uses the condition that $D$ is scalar flat). When $K_X$ is proportional to $L$, Conjecture \ref{conj1-1} holds by the results of \cite{OdakaSun, Berman, LiSun}. The conjecture can also be intuitively interpreted as a form of \emph{inversion of adjunction} for K-stability, if one assumes the Yau-Tian-Donaldson conjecture holds in dimension $n-1$. It is an interesting question to ask if the algebro-geometric counterpart can be proved directly. From the differential geometric point of view the conjecture also suggests the existence of complete K\"ahler metrics with negative constant scalar curvature on the complement $X\setminus D$, which is related to the work of H. Auvray \cite{Auvray}. 

In this paper we will deal with the case $n=1$, so $X$ is a smooth Riemann surface, and $D=\sum_{i=1}^d p_i$ is an effective divisor of degree d (all $p_i$'s are distinct). We call such pair $(X, D)$ a \emph{log Riemann surface}. The condition $\sigma<0$ is equivalent to $d> \chi(X)$. Conjecture \ref{conj1-1} in this case follows from the aforementioned results. However, the proofs in \cite{OdakaSun, Berman, LiSun} all depend crucially on the special feature that the canonical bundle of $X$ is definite, so seem difficult to be adapted to the general case.  Our proof here is based on the quantization technique and reveals the relationship between logarithmic K-stability and the known complete hyperbolic metric on $X\setminus D$. We hope the techniques developed in this paper could help understand the quantization for other types of singular metrics, for examples, those with cone singularities, and lead to the proof that  existence of singular cscK/extremal metrics with prescribed asymptotic behavior implies an appropriately extended notion of K-stability. For  metrics with cone singularities or Poincar\'e type singularities along a divisor this has already  been speculated in \cite{Donaldson2010, Szekelyhidi}. 

Before stating our main result,  we recall some known facts and fix some notation.  Let $V$ be a subvariety of $\C\P^N$ and $W$ a subvariety of $V$. For $\lambda\in [0, 1]$ we define the \emph{$\lambda$-center of mass} of $(V, W)$ to be 

$$\mu(V, W, \lambda)=\lambda \int_V \frac{ZZ^*}{|Z|^2} d\mu_{FS}+(1-\lambda)\int_W \frac{ZZ^*}{|Z|^2}d\mu_{FS}-\frac{\lambda Vol(V)+(1-\lambda)Vol(W)}{N+1}Id $$
where $[Z]\in \C\P^N$ is viewed as a column vector, and the volume is calculated with respect to the induced Fubini-Study metric. Notice $\mu$ always takes value in $\sqrt{-1}Lie(SU(N+1))$; indeed, by general theory $\mu$ can be viewed as the moment map for the action of $SU(N+1;\C)$ on a certain \emph{Chow variety}. For $B\in Lie(SU(N+1))$  we write $\|B\|_2:=\sqrt{Tr BB^*}$.

A pair $(V, W)$ embedded in $\C\P^N$ with vanishing $\lambda$-center of mass is called a \emph{$\lambda$-balanced embedding}. We say $(V, W)$ is  \emph{$\lambda$-Chow stable} if there is an $A\in SL(N+1;\C)$ such that $(A.V, A.W)$ is $\lambda$-balanced.  and we say $(V, W)$ is \emph{$\lambda$-Chow semistable} if the infimum balancing energy 
$$E(V, W, \lambda):=\inf_{A\in SL(N+1;\C)} \|\mu(A. V, A. W, \lambda)\|_2$$
vanishes. It is well-known that by the Kempf-Ness theorem, these definitions agree with the usual notion of Chow (semi)-stability of log pairs (see for example \cite{LW}). When $\lambda=1$ the subvariety $W$ can be ignored and this reduces to the standard notion of Chow (semi)-stability. 

Now going back to our situation of a polarized manifold $(X, L)$ and a smooth divisor $D$. We say $(X, D, L)$ is $\lambda$-\emph{almost asymptotically Chow stable} if for $k$ sufficiently large, under the projective embedding of $(X, D)$ induced by sections of $H^0(X, L^k)$ we have $E(V, W, \lambda)=o(k^{-1+n/2})$. By \cite{ssun} if $(X, D, L)$ is $\lambda$-amost asymptotically Chow stable then $(X, D, L, \beta)$ is K-semistable for $\beta=\frac{3\lambda-2}{\lambda}.$  We will not explicitly make use of the notion of (logarithmic) K-(semi)stability in this article, so we will not elaborate on the definition and we refer the readers to \cite{ssun}.

Restricting to our setting of a log Riemann surface $(X, D)$, with an ample line bundle $L$ of degree $l$. In this one dimensional case we do not need to assume $L=[D]$. We denote by $\omega$ the complete K\"ahler metric on $X\setminus D$ with constant nagative curvature, with total volume $2l\pi$, and with standard cusp singularities at points in $D$. $\omega$ can be considered as a closed \kahler current on $X$ whose cohomology class is $2\pi c_1(L)$. So there is a singular metric $h$ on $L$ such that the curvature of $h$ is $\omega$. For $k$ large, we denote by  $\mathcal H_k$ the subspace of $H^0(X, L^k)$ consisting of holomorphic sections that are $L^2$ integrable with respect to the norm defined by $h$ and $\omega$. It is easy to see that $\mathcal H_k$ agrees with the image of the map $H^0(X, L^{k}(-D))\rightarrow H^0(X, L^k)$  given by multiplication  by the defining section $s_D$ for $D$.  For $k$ large, we have an embedding $\Phi_k: X\rightarrow\P\mathcal H_k^*$. A choice of orthonormal basis of $\mathcal H_k$ determines a Hermitian isomorphism of $\P\mathcal H_k^*$ with $\C\P^{N_k}$, unique up to the  $U(N_k+1)$ action, where $N_k+1=\dim \mathcal H_k$. In particular, the quantity $\|\mu(X, D, \lambda)\|_2$ is independent of the choice of orthonormal basis.   The following is our main result

\begin{thm} \label{thm1-1}
Given a log Riemann surface $(X, D)$ with $d> \chi(X)$, and any ample line bundle $L$ over $X$, we have for $k$ large, 
$$\|\mu(\Phi_k(X), \Phi_k(D), \frac 23)\|^2_2=O(k^{-{3/2}}(\log k)^{121}). $$
\end{thm}

We remark here that the exponent $121$ is far from being optimal, and it can certainly be improved when needed. We can \emph{roughly} say $(X, D, L)$ is $\frac{2}{3}$-almost asymptotically Chow stable in the sense of above definition. This is not precisely true since our embedding using $L^2$ sections is not induced by the complete linear system $|L^k|$. However, one can always construct from this an almost balanced embedding induced by the complete linear system. Fix a splitting $H^0(X, L^k)^*=\mathcal H_k^*\bigoplus \oplus_{p\in D} \C_{p}$, where $\C_p$ is the one dimensional subspace of $H^0(X, L^k)^*$ defined by evaluating a section at $p$.  Then we extend the $L^2$ metric on $\mathcal H_k$ to a Hermitian metric on $H^0(X, L^k)$ such that the different pieces are orthogonal. Now we define a pair of cycles $(X', D')$ in $\P(H^0(X, L^k)^*)$, where $D'$ is the union of points corresponding to the one dimensional subspaces $\C_p$, and $X'$ is the union of the image of $\Phi_k(X)$, together with all the lines that connect the image of a point $p\in D$ in $\P\mathcal H_k^*$ and the corresponding point in $D'$. This cycle  is in the closure of the $PGL$ orbit of the pair $(X, D)$ embedded by the linear system $|L^k|$; it is indeed given by deformation to the normal cone, see analogous discussion in Section 4). Since the center of mass of a line is easy to compute, it is then not hard to check that $(X', D')$ is indeed $\frac{2}{3}$-almost balanced, which implies $(X, D)$ also admits a $\frac{2}{3}$-almost balanced embedding. 

As an immediate corollary, using \cite{ssun}, is that 

\begin{cor}\label{1.1}
$(X, D, L, 0)$ is logarithmic K-semistable.
\end{cor}

\begin{rem}
We mention that corollary \ref{1.1} were also proved in \cite{LW}, using explicit Hilbert-Mumford criterion and the special feature in complex dimension one. As mentioned above, the main interest in our paper is indeed the quantitative estimate of the balancing energy of the $L^2$ embedding induced by the hyperbolic metric. We hope this will have applications in higher dimensions.
\end{rem}

Now we briefly describe the idea involved in  the proof of Theorem \ref{thm1-1}.  Let $\{s_\alpha\}$ be an orthonormal basis of $\mathcal H_k$. An important quantity is the ``\emph{density of state function}" (or the \emph{Bergman kernel function}) 
$$\rho_k=\sum_i|s_\alpha|^2_h. $$
 Denote $\omega_k=\Phi_k^*\omega_{FS}$ (here our convention is that $\omega_{FS}\in c_1(O(1))$, then 
 $$2\pi \omega_k= k\omega+i\p\bp \log \rho_k. $$
 We know by definition 
\begin{equation} \label{eqn1-1}
\int_X \langle s_\alpha, s_\beta\rangle_h \omega=\delta_{\alpha\beta}, 
\end{equation}
and we can write 
\begin{equation} \label{eqn1-2}
\mu(X, D, \lambda)=\lambda \int_X  \langle s_\alpha, s_\beta \rangle_h \rho_k^{-1}\omega_k+(1-\lambda)\sum_\alpha \rho_k(p_i)^{-1} \langle s_\alpha(p_i), s_\beta(p_i)\rangle_h-c_k I,
\end{equation}
where $c_k=\frac{\lambda kl+(1-2\lambda )d}{N_k+1}$. 
Since $D$ consists of finitely many points, the key is to understand the first term of (\ref{eqn1-2}). Compared with (\ref{eqn1-1}),  it is then important to know the bahavior of $\rho_k^{-1}\omega_k$. Not surprisingly, as in the case without divisor, we need to study the function $\rho_k$. 

If $\omega$ were a smooth K\"ahler metric on $X$, it would follow from the result of Tian, Zelditch, Lu\cite{Tian1990On, Zelditch2000Szego, Lu2000On, Catlin, MM}, that $\rho_k$ has an asymptotic expansion of the form

\begin{equation} \label{eqn1-3}
\rho_k=\frac{1}{2\pi}[k+\frac{S(\omega)}{2}+O(k^{-1})], 
\end{equation}

Now as observed in \cite{Donaldson15} this result can be localized. The basic point is that for any $p\in X$ away from $D$, we have 
$$\rho_k(p)=\sup\{|s(p)|^2|s\in \mathcal H_k, \|s\|=1\}. $$
 and the supreme is achieved by a so-called \emph{peak section}. When $k$ is sufficiently large, the rescaled manifold $(X, p, L^k, h^{\otimes k}, k\omega)$ is close to the standard Gaussian model $(\C, 0, L_0, h_0, \omega_0)$, where $L_0$ is the trivial line bundle over $\C$, $h_0$ is the (non-trivial) hermitian metric $e^{-|z|^2/2}$ whose curvature $\omega_0$ is the standard flat metric. This fact allows a construction of the peak section at $p$ by a grafting and perturbation procedure. Everything is local in $p$ except the perturbation involves H\"ormander's $L^2$ estimate which depends on the global lower bound of Ricci curvature (this is automatically satisfied in our case). A more careful analysis shows that the expansion (\ref{eqn1-3}) indeed holds for points $p$ whose injectivity radius is bounded below by $k^{-1/2}\log k$. 

The new feature arises when we want to understand $\rho_k$ at the points with small injectivity radius, i.e. points very close to $D$ (in the standard topology on $X$). A difficult point here is that $D$ has co-dimension one. If $D$ were of higher co-dimension, then by the result of \cite{Donaldson15} one can ignore a neighborhood of $D$ and obtain a better estimate than that is stated in Theorem \ref{thm1-1}.

 Notice that $\rho_k$ goes to zero near $p_i$, so we can not expect the same expansion as (\ref{eqn1-3}) to hold. Instead we need to look at a different model, which is the punctured hyperbolic disc $\mathbb D^*$. In Section 2 we will analyze the behaviors of $\rho_k$ and $\rho_k^{-1}\omega_k$ in the model case. As our investigation shows, there are also two further distinct behavior according to the size of the injectivity radius. For a point $p$ with injectivity radius smaller than $k^{-1/2}(\log k)^{-1}$ we show that the function $\rho_k$ is essentially governed by at most three monomial sections (so we can intuitively think of these as sections ``peaked" around a circle instead of at one point); for a point with injectivity radius between $[k^{-1/2}(\log k)^{-1}, k^{-1/2}\log k]$, there are infinitely many monomial sections contributing to $\rho_k$, and we need to do a much more careful analysis to get the required estimates. 
 
 In Section 3 we will use the results of Section 2 to prove Theorem \ref{thm1-1}. In Section 4, we will study for the case $X=\P^1$ and $L=O(d)$, the exact range of $\lambda\in [0, 1]$ for which $(X, D)$ is $\lambda$-Chow stable under the embedding induced by $L^k$.  The key point is that for the minimum such $\lambda$, which we denote $\lambda_k$, we need to construct a degeneration of $(X, D)$ to a $\lambda_k$-balanced pair $(X_0, D_0)$.  We will show that $\lambda_k<2/3$ and prove the existence of such pair. It turns out  that the degeneration is exactly given by deformation to the normal cone of $D$, so $X_0$ consists of $d+1$ components, one isomorphic to $X$, and the other are $d$ lines.  This contrasts the case considered in \cite{ssun} (see also Figure \ref{fig1} and Figure \ref{fig2}), when $\sigma=0$ (in the one dimensional case, this means $X=\P^1$ and $D$ consists of two points). In that case the limiting balanced pair consists of a chain $X_0$ of $k$ lines in $\P^k$, so the number of components goes to infinity as $k$ tends to infinity. One would expect the same picture to also hold in higher dimension. This difference should also reflect the interesting facts that the complete negative K\"ahler-Einstein metrics constructed in \cite{ChengYau, TianYau1, Kobayashi} has finite volume, while the Tian-Yau complete Ricci-flat metric constructed in \cite{TianYau2} has infinite volume. 
 
 The draft of this paper was finished around November 2015. Just before the first version of this paper was posted in arXiv.org we were informed of the paper by Auvray-Ma-Marinescu \cite{AMM}, which studies the Bergman kernels on punctured Riemann surfaces. There are also many results in the literature studying the asymptotics of Bergman kernels of singular K\"ahler metrics, see for example  \cite{LL, RT, DLM}. Our paper has different motivation from these and for our geometric purpose we need more refined information of the Bergman kernel than the other works quoted above.

\

\textbf{Acknowledgements.}  We would like to thank Professor Simon Donaldson for insightful discussions regarding quantization of K\"ahler metrics over long time, and we are grateful to Professors Xiuxiong Chen, Dror Varolin and Bin Xu for their interest in this result. This project started after the talk by the second author in the  workshop ``Quantum Geometry, Stochastic Geometry, Random Geometry, you name it" in the Simons Center in June 2015, and he thanks Steve Zelditch for the invitation. The first author would also like to thank Professor Xiuxiong Chen for the hospitality while his stay in USTC, and he is always grateful to Professor Bernard Shiffman for his continuous and unconditional support.

\section{Calculation on the model}  
Throughout this paper we will denote by $\epsilon (k)$ a quantity depending on $k$ that is $O(k^{-m})$ as $k\rightarrow\infty$, for all $m\geq0$. 
Recall that our model is the punctured disk $\mathbb D^*=\{z\in \C||z|\leq 1\}$, endowed with the K\"ahler metric 
 \begin{equation} \label{eqn2-1}
 \omega_0=\frac{i dz\wedge d\bar{z}}{|z|^2(\log \frac{1}{|z|^2})^2}.
 \end{equation}
The corresponding K\"ahler potential is $\Phi_0=-\log \log \frac{1}{|z|^2}$, and the scalar curvature of $\omega_0$ is $-2$. For $k\geq 1$, we let $\mathcal H_{k, 0}$ be the Bergman space of holomorphic functions $f$ on $\mathbb D^*$ such that 
$$\|f\|_k^2:=\int_{\mathbb D^*} |f|^2 e^{-k\Phi_0}\omega_0<\infty. $$
On $\mathcal H_{k, 0}$ we denote by $\langle\cdot, \cdot\rangle_k$ the corresponding Hermitian inner product.

\begin{lem}\label{integral} For any $a\geq 1$, we have $z^a\in \mathcal H_{k, 0}$ and 
\begin{equation} \label{eqn2-2}
\langle z^a, z^b\rangle_k=\frac{2\pi(k-2)!}{ a^{k-1}}\delta_{ab}. 
\end{equation}
In particular, the functions $\{(\frac{ a^{k-1}}{2\pi(k-2)!})^{1/2}z^a|a\geq 1\}$  form an orthonormal basis of $\hcal_{k,0}$
\end{lem}
\begin{proof}
First of all, it is easy to see that $z^a$ is $L^2$ integrable with respect to the given weight if and only if $a\geq 1$.
By the $S^1$ symmetry of the metric and the weight,  $z^a$'s are obviously orthogonal to each other. Now we calculate the norms:
\begin{eqnarray*}
\|z^a\|^2_{k}&=&\int_{\mathbb D^*}|z|^{2a}(\log \frac{1}{|z|^2})^k\frac{idzd\bar{z}}{|z|^2(\log \frac{1}{|z|^2})^2}=\int_{\mathbb D^*}|z|^{2(a-1)}(\log \frac{1}{|z|^2})^{k-2}idzd\bar{z}\\
&=&2\pi \int_0^1x^{a-1}(-\log x)^{k-2}dx
\end{eqnarray*}
Using the substitution $t=-\log x$, we get $$\|z^a\|^2_{k}=2\pi\int _0^{\infty}t^{k-2}e^{-at}dt=\frac{2\pi(k-2)!}{ a^{k-1}}$$
\end{proof}

\begin{rem}
The calculation above actually shows us more. By the substitution $t=(k-2)y$, we get that 
 $$\int _0^{\infty}t^{k-2}e^{-at}dt=(k-2)^{k-1}\int_0^{\infty}e^{(k-2)(\log y-ay)}dy.$$
  So Laplace's method tells that for large $k$ the integral is concentrated in a small neighborhood of $t=\frac{k-2}{a}$, i.e. $|z|^2=e^{-(k-2)/a}$ (see figure \ref{fig0}). Moreover, the concentration is within a neighborhood of radius $\frac{k^{1/2}\log k}{a}$ of $t=\frac{k-2}{a}$, i.e. 
  \begin{equation} \label{eqn2-2-2}
  \int_{|t-\frac{k-2}{a}|\leq \frac{k^{1/2}\log k}{a}} t^{k-2}e^{-at}dt\geq (1-\epsilon(k)) \int _0^{\infty}t^{k-2}e^{-at}dt. 
  \end{equation}
  Here the error term $\epsilon(k)$ is, to be more precise, less than $k^{-k/3}$, which is independent of $a$. 
  \begin{figure}
  	\begin{center}
  		\includegraphics[width=0.8 \columnwidth]{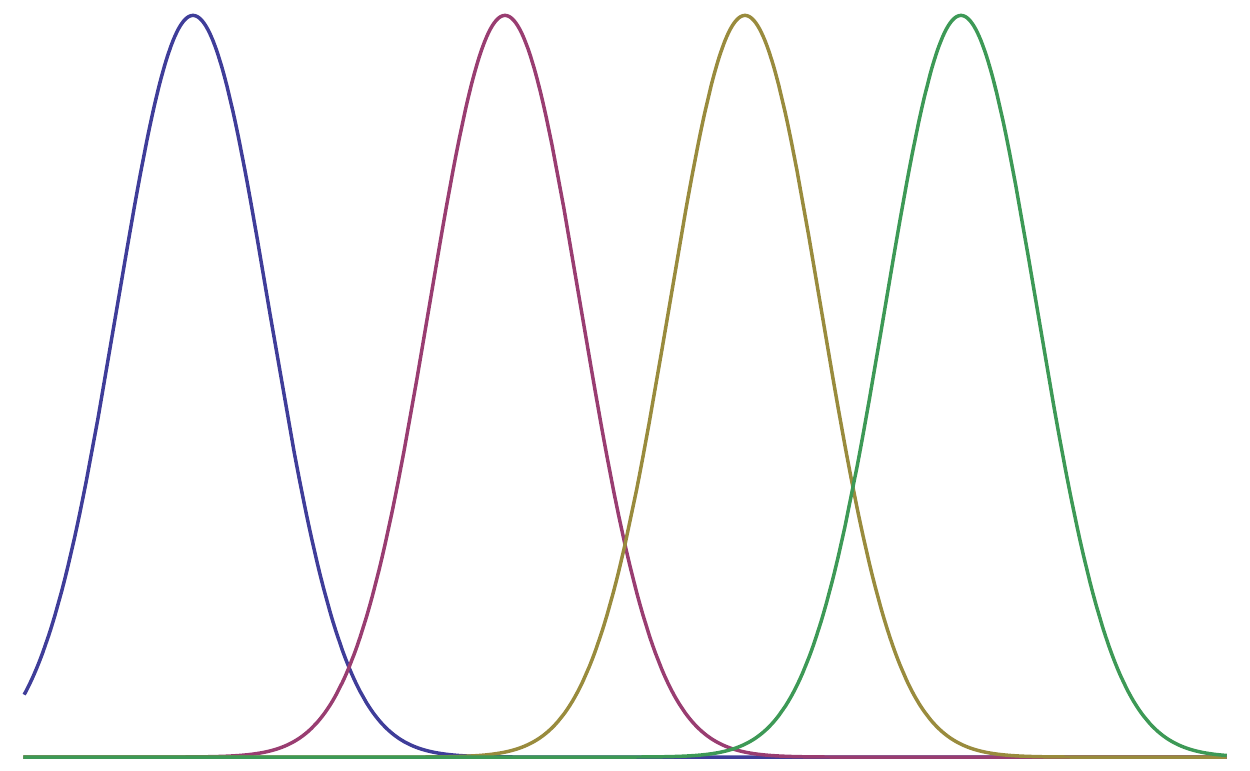}
  		\caption{Mass Concentration}
  		\label{fig0}
  	\end{center}
  \end{figure}
\end{rem}

From the Lemma above it follows that the Bergman kernel of $\hcal_{k,0}$ is given by 
\begin{equation} \label{eqn2-3}
\rho_{k,0}=\frac{(\log 1/|z|^2)^k}{2\pi(k-2)!}\sum_{a=1}^{\infty}a^{k-1}|z|^{2a}.
\end{equation}
By the preceding remark, we see that near the origin, only those terms of small degrees matter. So we can heuristically view $\rho_{k, 0}$ as a polynomial function in $|z|^2$. Formally the above orthonormal basis of $\mathcal H_{k, 0}$ induces an embedding of $\mathbb D^*$ into an infinite dimensional complex projective space, and the pull-back of the Fubini-Study metric is given by
 
\begin{equation}\label{eqn2-4}
\omega_{k, 0}:=\frac{1}{2\pi}(k\omega_0+i\p\bp \log \rho_{k, 0})=\frac{1}{2\pi}i\partial\bar{\partial}\log \sum _{a=1}^{\infty}a^{k-1}|z|^{2(a-1)}. 
\end{equation}
 
Our main goal in this section is to understand $\rho_{k, 0}$ and $\omega_{k, 0}$. This serves as a local model for understanding the Bergman kernel and the induced Fubini-Study metric near a hyperbolic cusp in our setup described in the introduction.

To simplify notations we will denote $x=|z|^2$, and we will shift $k$ by $1$ (so we are studying $\hcal_{k+1, 0}$ instead). 
We write
$$\phi_k(x)=\sum_{a=1}^{\infty}a^kx^{a-1}, $$
then 
$\omega_{k+1}(x)=\frac{1}{2\pi}\phi_k^{-2}\psi_k idz\wedge\bar z, $ where  
$$\psi_k=\phi_k \Delta_z\phi_k-|\p_z\phi_k|^2 $$
The integral in the model case corresponding to the one we are interested in (\ref{eqn1-2}) is the following
\begin{equation} \label{eqn2-5}
\mu_a:=\int_{\mathbb D^*} (\frac{|z^a|}{\|z^a\|_{k+1}})^2 e^{-(k+1)\Phi_0}\rho_{k+1, 0}^{-1} \omega_{k+1}=\int_0^1 x^{a-1}a^k\phi_k(x)^{-3}\psi_k(x)dx. 
\end{equation}

Similarly to the compact case, to measure the deviation of the image of $\mathbb D^*$ in the infinite dimensional projective space from being $\frac{2}{3}$-balanced, we need to estimate $\frac{2}{3}\mu_a+\frac{1}{3}\delta_{a1}$.   We divide into three cases

\

\textbf{Case I}: $a\geq k^{1/2}\log k$. In this case by the remark above $z^a$ is concentrated around the points where $|z|^2$ approximately $e^{-k^{1/2}(\log k)^{-1}}$. The injectivity radius of the metric $\omega_0$ at these points is approximately $\pi(\log\frac{1}{|z|^2})^{-1}\approx  \pi k^{-1/2}\log k $. Then as mentioned in the introduction, when $|z|^2 \geq e^{-k^{1/2}(\log k)^{-1}}$,  the usual proof of the Bergman kernel expansion (c.f. \cite{Donaldson15}) goes through, and provides an uniform estimate
\begin{equation} \label{usualexpansion}
\rho_{k, 0}=\frac{1}{2\pi} (k-1+O(k^{-1})), 
\end{equation}
which holds in the $C^2$ norm. This implies 
 $$\omega_{k, 0}=\omega_0(1+O(k^{-2})),  $$
So we obtain that $\mu_a=1+O(k^{-2})$ for $a\geq k^{1/2}\log k$. 
Moreover,  this argument also gives rise the following estimate of the volume of $\omega_{k, 0}$. 

\begin{lem} \label{lem2-2}
\begin{equation}
\int_{|z|^2\leq e^{-k^{1/2}(\log k)^{-1}}} \omega_{k, 0}=O(k^{-1/2}\log k). 
\end{equation}
\end{lem}

\begin{proof}
By definition $\omega_{k, 0}=\frac{1}{2\pi} k\omega_0+\frac{1}{2\pi}i\p\bp \log \rho_{k, 0}$. A direct calculation shows, 
$$\int_{|z|^2\leq e^{-k^{1/2}(\log k)^{-1}}} \omega_0=k^{-1/2}\log k. $$
For the other term, using integration by parts and the above expansion of $\rho_{k, 0}$, we have
$$|\int_{|z|^2\leq e^{-k^{1/2}(\log k)^{-1}}}  i\p\bp \log \rho_{k, 0}|\leq |\int_{|z|^2=e^{-k^{1/2}(\log k)^{-1}}} Jd\rho_{k, 0}|=O(k^{-2}\log k). $$
\end{proof}

\

\textbf{Case II}: $a=o(k^{1/2}(\log k)^{-1/2})$. In this case the sections $z^a$ are concentrated in a very small annular neighborhood of $0$. We have
\begin{eqnarray*}
\psi_k(x)
&=&\sum_{a=1,b=2}a^kb^k(b-1)^2x^{a+b-3}-\sum_{a=2,b=2}a^kb^k(a-1)(b-1)x^{a+b-3}\\
&=&\sum_{b=2}b^k(b-1)^2x^{b-2}+\sum_{a=2,b=2}a^kb^k(b-1)(b-a)x^{a+b-3}\\
&=&\sum_{l=3}c_lx^{l-3}, 
\end{eqnarray*}

where \begin{eqnarray*}
c_l&=&\sum_{a+b=l,a\geq 2, b\geq 2}a^kb^k(b-1)(b-a)+(l-2)^2(l-1)^k\\&=&\sum_{a+b=l,a\geq 1, b\geq 2}a^kb^k(b-1)(b-a)\\&=&\sum_{a+b=l,a\geq 1, a<b}a^kb^k(b-a)^2
\end{eqnarray*} 

As power series, $\phi_k$ and $\psi_k$ are complex for integrals. Our first observation is that $\phi_k(x)$ can be estimated using only 2 or 3 terms when $x$ is small. More precisely:

\begin{lem}\label{l1}
When $x\in [0, 2^{-k}]$, $\phi_k(x)=1+2^kx+\epsilon(k)$. Similarly, for $x\in [\frac{n+1}{n})^{-k}, (\frac{n+2}{n+1})^{-k}]$, we have 
$$\phi_k(x)=(1+\epsilon(k))(n^kx^{n-1}+(n+1)^kx^n+(n+2)^kx^{n+1})$$
 as long as $n^2=o(\frac{k}{\log k})$.
\end{lem}
\begin{proof}
The quotient of two adjacent  terms is $\frac{(a+1)^kx}{a^k}$.  Suppose $x\leq 2^{-k}$. Notice $\frac{a+1}{2a}<\frac{3}{4}$ for $a\geq 3$. So
$$\sum_{a\geq 3}a^kx^{a-1}\leq (\frac{3}{4})^k\sum_{a\geq 0}(\frac{3}{4})^a=\epsilon(k).$$ 
Now suppose $(\frac{n+1}{n})^{-k}\leq x\leq (\frac{n+2}{n+1})^{-k}$ for some integer $n$. For $a\geq n+2$, we have 
$$\frac{(a+1)(n+1)}{a(n+2)}\leq 1-\frac{1}{(n+2)^2}. $$
So as long as $n^2=o(\frac{k}{\log k})$, $\frac{(a+1)^kx}{a^k}=\epsilon(k)$. Then
$$\sum_{a\geq n+3}a^kx^{a-1}=(1+\epsilon(k))(n+2)^kx^{n+1}). $$
 Similarly, for $a\leq n-1$, we have 
 $$\sum_{a\leq n-1}a^kx^{a-1}=(1+\epsilon(k))n^kx^{n-1}). $$ 
 The Lemma is then proved. 
\end{proof}

Now we consider $\psi_k(x)$. As in the proof of the preceding Lemma, we first notice that $c_l$ is dominated by the middle terms as long as $l^2=o(\frac{k}{\log k})$. More precisely, when $l$ is odd, 
$$c_l=(1+\epsilon(k))(\llcorner\frac{l}{2}\lrcorner\ulcorner\frac{l}{2}\urcorner)^k$$ 
where $\llcorner\lrcorner$ and $\ulcorner\urcorner$ means round-down and round-up respectively. When $l$ is even, 
$$c_l=(1+\epsilon(k))4\cdot ((l/2-1)(l/2+1))^k. $$

With these in mind,  we can now approximate $\psi_k(x)$. More precisely, we have
\begin{lem}\label{l2}
When $x\leq (\sqrt{3})^{-k}$, $\psi_k(x)=(1+\epsilon(k))(2^k+6^kx^2)$. Similarly, for $x\in [(\sqrt{\frac{n+2}{n}})^{-k},  (\sqrt{\frac{n+3}{n+1}})^{-k}]$, we have 
$$\psi_k(x)=(1+\epsilon(k))((n(n+1))^kx^{2(n-1)}+((n+1)(n+2))^kx^{2n}+((n+2)(n+3))^kx^{2(n+1)})$$ as long as $n^2=o(\frac{k}{\log k})$.
\end{lem}
\begin{proof}
The proof is similar to that for $\phi_k(x)$. We only want to remind the reader that the odd power terms are omitted. The reason is that within each interval appeared in the Lemma the odd power terms are dominated by the adjacent even power terms.
\end{proof}

Lemma \ref{l1} describes a set of ladders $a_n^{-k}$ for $\phi_k(x)$, where $a_n=\frac{n+1}{n}$, $n \geq 1$. Lemma \ref{l2} describes a set of ladders $b_n^{-k}$ for $\psi_k(x)$, where $b_n=\sqrt{\frac{n+2}{n}}$, $n \geq 1$. One immediately sees that $$a_n>b_n>a_{n+1}$$
Since the integral we are interested in involves both $\phi_k(x)$ and $\psi_k(x)$ and we want to use the approximations given by Lemma \ref{l1} and lemma \ref{l2}, we will further refine our intervals to that of the form $(a_n^{-k},b_n^{-k})$ and $(b_n^{-k},a_{n+1}^{-k})$. The following is a direct consequence of Lemma \ref{l1} and Lemma \ref{l2}. 

\begin{lem}\label{intervals}
\begin{itemize}
\item Within the interval $[a_n^{-k},b_n^{-k}]$, we have 
\begin{eqnarray*}
\phi_k(x)&=&(1+\epsilon(k))(n^kx^{n-1}+(n+1)^kx^n)\\
\psi_k(x)&=&(1+\epsilon(k))((n(n+1))^kx^{2(n-1)}+((n+1)(n+2))^kx^{2n})
\end{eqnarray*}
\item Within the interval $[b_n^{-k},a_{n+1}^{-k}]$, we have 
\begin{eqnarray*}
\phi_k(x)&=&(1+\epsilon(k))((n+1)^kx^n+(n+2)^kx^{n+1})\\
\psi_k(x)&=&(1+\epsilon(k))((n(n+1))^kx^{2(n-1)}+((n+1)(n+2))^kx^{2n})
\end{eqnarray*}
\end{itemize}
\end{lem}

Now we are ready to evaluate integrals. 

\begin{prop}
\begin{itemize}
\item[i)] When $a=1$, we have 
\begin{eqnarray*}
\int_0^{2^{-k}}\frac{\psi_k(x)dx}{(\phi_k(x))^3}=\frac{3}{8}+\epsilon(k)\\
\int_{2^{-k}}^{(\sqrt{3})^{-k}}\frac{\psi_k(x)dx}{(\phi_k(x))^3}=\frac{1}{8}+\epsilon(k)
\end{eqnarray*}
\item[ii)] When $a=2$, we have 
\begin{eqnarray*}
\int_0^{2^{-k}}2^kx\frac{\psi_k(x)dx}{(\phi_k(x))^3}=\frac{1}{8}+\epsilon(k)\\
\int_{2^{-k}}^{(\sqrt{3})^{-k}}2^kx\frac{\psi_k(x)dx}{(\phi_k(x))^3}=\frac{3}{8}+\epsilon(k)\\
\int_{(\sqrt{3})^{-k}}^{(3/2)^{-k}}2^kx\frac{\psi_k(x)dx}{(\phi_k(x))^3}=\frac{3}{8}+\epsilon(k)\\
\int_{(3/2)^{-k}}^{(\sqrt{2})^{-k}}2^kx\frac{\psi_k(x)dx}{(\phi_k(x))^3}=\frac{1}{8}+\epsilon(k)
\end{eqnarray*}
\end{itemize}
\end{prop}
\begin{proof}
For each integral, we replace $\phi_k(x)$ and  $\psi_k(x)$ with the corresponding approximations listed in lemma \ref{intervals}. Then by simple substitutions, we can evaluate the integrals.
\end{proof}
The picture we see for $a=2$ actually reflects the picture for general $a$. We will use the following notations:
\begin{eqnarray*}
I_{a,n}=\int_{a_n^{-k}}^{b_n^{-k}}(a+1)^kx^a\frac{\psi_k(x)dx}{(\phi_k(x))^3}\\
I_{a,n}'=\int_{b_n^{-k}}^{a_{n+1}^{-k}}(a+1)^kx^a\frac{\psi_k(x)dx}{(\phi_k(x))^3}
\end{eqnarray*}
\begin{prop}For $n\geq 2$
\begin{eqnarray*}
I_{n-1,n}=\frac{1}{8}+\epsilon(k)\\
I_{n,n}=\frac{3}{8}+\epsilon(k)\\
I_{n,n}'=\frac{3}{8}+\epsilon(k)\\
I_{n+1,n}'=\frac{1}{8}+\epsilon(k)
\end{eqnarray*}
\end{prop}
\begin{proof}
Plugging in the approximations for $\phi_k(x)$ and $\psi_k(x)$, we get
$$I_{a,n}=\int_{a_n^{-k}}^{b_n^{-k}}(\frac{(a+1)(n+1)}{n^2})^kx^{a-n+1}\frac{1+(\frac{n+2}{n+1})^kx^2}{(1+(\frac{n+1}{n})^kx)^3}dx$$
and 
$$I_{a,n}'=\int_{b_n^{-k}}^{a_{n+1}^{-k}}(\frac{(a+1)n}{(n+1)^2})^kx^{a-n-2}\frac{1+(\frac{n+2}{n+1})^kx^2}{(1+(\frac{n+2}{n+1})^kx)^3}dx$$
When $a=n$, we use the substitution $y=(\frac{n+1}{n})^kx$, and get 
\begin{eqnarray*}
I_{n,n}&=&\int _1^d\frac{x(1+bx^2)}{(1+x)^3)}dx\\
&=&\frac{1+b}{2(1+x)^2}-\frac{1+3b}{1+x}+b(1+x)-3b\log (1+x)|^d_1\\
&=&(\frac{1}{2(1+x)^2}-\frac{1}{1+x})|_1^d+\epsilon(k)\\
&=&\frac{3}{8}+\epsilon(k),
\end{eqnarray*}
where $b=(\frac{(n+2)n}{(n+1)^2})^k$ and $d=\sqrt{\frac{1}{b}}$.

We can compute the other 3 intergrals in the same way, using the fact the integrands are all rational functions. 
\end{proof}

In order to calculate $\mu_a$, we need also calculate the integrals on other intervals. The following lemma tells us that we already have the main value.
\begin{prop}\label{goodrem}
For $n\geq 2$
\begin{eqnarray*}
I_{n-2,n}=\epsilon(k)\\
I_{n+1,n}=\epsilon(k)\\
I_{n-1,n}'=\epsilon(k)\\
I_{n+2,n}'=\epsilon(k)
\end{eqnarray*}
\end{prop}
The calculations are basically the same as that in the last proposition. This proposition shows that the integrals on the nearby intervals are negligible.
 As we have remarked, the mass of the integrands decay rapidly away from the main intervals. More precisely, we can write 
 $$\mu_n=\int_0^1 \frac{n^kx^{n-1}}{\phi_k(x)}\omega_{k, 0}. $$ 
 We claim the contribution of the integral from $x\leq a_{n-1}^{-k}$ or $x\geq a_{n+1}^{-k}$ are both $\epsilon(k)$. Since $a=o(k^{1/2}(\log k)^{-1})$, we know from the definition of $\phi_k(x)$ that the integrand itself is $\epsilon(k)$ for $x$ in this region. Now by Lemma \ref{lem2-2}, it follows that the contribution from $(0, a_{n-1}^{-k})$ and $(a_{n+1}^{-k}, e^{-k^{1/2}(\log k)^{-1}})$ to $\mu_n$ is  $\epsilon(k)$. When $x>e^{-k^{1/2}(\log k)^{-1}}$, we know from Case I that $\omega_{k, 0}\leq 2\rho_{k,0}\omega_0$. But since 
 $$\int \frac{n^kx^{n-1}}{\phi_k(x)}\rho_{k,0}\omega_0=1, $$ 
and  by (\ref{eqn2-2-2}) we know the contribution to this integral from $x>e^{-k^{1/2}(\log k)^{-1}}$ is $\epsilon(k)$, so the contribution to $\mu_n$ is also $\epsilon(k)$. This proves the claim. 
 
 \
 
Therefore we obtain

\begin{theorem}
For $a>1$ satisfying $a^2=o(\frac{k}{\log k})$, we have $\mu_a=1+\epsilon(k)$. When $a=1$, we have $\mu_1=\frac{1}{2}+\epsilon(k)$
\end{theorem}

\

\textbf{Case III}: $a\in [k^{1/2}(\log k)^{-1}, k^{1/2}\log k]$. In this case the sections $z^a$ are concentrated in the ``neck region". 
In order to estimate $\mu_a$, we will compare it with the above standard integral. We  may write $$\gamma_a(x)=\frac{\phi_k(x)}{a^kx^{a-1}}, $$  then
$$\mu_a=\int_0^1 \frac{\Delta_z\log \gamma_a(x)dx}{\gamma_a(x)}.$$
 Next we use substitution $v=\log \frac{1}{x}$. Then $\gamma_a(x)=\sum_{c=-a+1}^{\infty} (\frac{a+c}{a})^ke^{-cv}$. Let $v=u+\frac{k}{a}$, we can write $$\gamma_a(x)=f_a(u):=\sum_{c\geq -a+1}e^{k(\log(1+\frac{c}{a})-\frac{c}{a} )}e^{-cu}, $$
and since $\Delta_z=\frac{1}{x}\frac{d^2}{du^2}$ and $dx=-xdu$. We get 
\begin{equation} \label{eqn2-6}
\mu_a=\int_{-k/a}^{\infty}\frac{(\log f_a(u))''du}{f_a(u)}
\end{equation}
\begin{lem}
\begin{equation}
\mu_a=\int_{-(\log k)^2}^{(\log k)^2}\frac{(\log f_a(u))''du}{f_a(u)}+\epsilon(k)
\end{equation}
\end{lem}
\begin{proof}
The idea again is to view $\frac{1}{f_a(u)}$ as the integrand,  with the measure part given by $\omega_{k, 0}$. We refer the reader to the arguments right after Proposition \ref{goodrem}, which also works here.

Since $f_a(u)$ is a convex function in $u$, and $f_a(0)\geq 1$,  we only need to show that when $|u|=(\log k)^2$, we have $\frac{1}{f_a(u)}=\epsilon(k)$. But this is already clear when we look at the terms when $|c|=1$.
\end{proof}

Now we have
$$\int_{-(\log k)^2}^{(\log k)^2}\frac{(\log f_a(u))''du}{f_a(u)}=\int_{-(\log k)^2}^{(\log k)^2}\frac{f_a''(u)du}{(f_a(u))^2}-\int_{-(\log k)^2}^{(\log k)^2}\frac{(f_a'(u))^2du}{(f_a(u))^3}$$

The following estimates are based on the simple fact that the function $\log (1+x)-x$ is concave with only a unique maximum at $x=0$. 

\begin{lem}\label{lem-gc}
$$\int_{-(\log k)^2}^{(\log k)^2}\frac{f_a''(u)du}{(f_a(u))^2}=2\int_{-(\log k)^2}^{(\log k)^2}\frac{(f_a'(u))^2du}{(f_a(u))^3}+\epsilon(k).$$
\end{lem}
\begin{proof}
This is basically integration by parts, since $$\int_{-(\log k)^2}^{(\log k)^2}\frac{f_a''(u)du}{(f_a(u))^2}=\frac{f'_a(u)}{(f_a(u))^2}|_{-(\log k)^2}^{(\log k)^2}+2\int_{-(\log k)^2}^{(\log k)^2}\frac{(f_a'(u))^2du}{(f_a(u))^3}.$$
So we need to evaluate the boundary values. 
When $u=(\log k)^2$, it is easy to see that both $f_a(u)$ and $f_a'(u)$ are dominated by the terms with $c\leq 0$. So $|\frac{f'_a(u)}{(f_a(u))^2}|\leq \frac{a}{f_a(u)}=\epsilon(k)$. Now we consider $u=-(\log k)^2$. 
Let $P(c)=k(\log(1+\frac{c}{a})-\frac{c}{a})-cu$. Then $P'(c)=\frac{k}{a+c}-\frac{k}{a}-u$ and $g''(c)=-\frac{k}{(a+c)^2}<0$. So $P(c)$ is a concave function of $c$. When $u=-(\log k)^2$, $f_a(u)$ and $f_a'(u)$ are dominated by the terms with $c>0$. The zero of $P'(c)$ is $c_0=\frac{k}{k/a+u}-a=\frac{-au}{\frac{k}{au}+1}$. So $f_a(u)$ is dominated by a term around $c_0=O((\log_k)^4)$ (since $c_0$ may not be an integer). Thus at $u=-(\log k)^2$, $|\frac{f'_a(u)}{(f_a(u))^2}|\leq \frac{|c_0|}{f_a(u)}=O(\frac{(\log_k)^4}{f_a(u)})=\epsilon(k)$. And the lemma is proved.
\end{proof}

From these we obtain 
$$\mu_a=\frac{1}{2}\int_{-(\log k)^2}^{(\log k)^2}\frac{f_a''(u)du}{(f_a(u))^2}+\epsilon(k). $$
We now further simplify the integral. Let 
$$g_a(u)=\sum_{|c|\leq (\log k)^5} e^{k(\log (1+\frac{c}{a})-\frac{c}{a})-cu}. $$
Then by similar arguments as in the proof of the previous lemma we  see that 
$$f_a(u)=g_a(u)(1+\epsilon(k)), $$
and 
$$f_a''(u)=g_a''(u)(1+\epsilon(k)). $$
So 
\begin{equation}
\mu_a=\frac{1}{2}\int_{-(\log k)^2}^{(\log k)^2}\frac{g_a''(u)du}{(g_a(u))^2}+\epsilon(k). 
\end{equation}

\

\noindent Now we define 
$$h_a(u)=\sum_{c\in \Z}e^{-\frac{kc^2}{2a^2}}e^{-cu}.$$
As before for $u\in [-(\log k)^2, (\log k)^2]$ we have
$$h_a(u)=(1+\epsilon(k))\sum_{|c|\leq (\log k)^5}e^{-\frac{kc^2}{2a^2}}e^{-cu}.$$

\begin{lem}
We have
$$\int_{-(\log k)^2}^{(\log k)^2}\frac{g_a''(u)du}{(g_a(u))^2}=\int_{-\infty}^{(\infty)^2}\frac{h_a''(u)du}{(h_a(u))^2}+O(\frac{(\log k)^{60}}{k}).$$
\end{lem}
\begin{proof}
For $u\in [-(\log k)^2, (\log k)^2]$, we have 
\begin{eqnarray*}
g_a(u)&=&\sum_{|c|\leq (\log k)^5} e^{-\frac{kc^2}{2a^2}-cu}(1+\frac{kc^3}{a^3}+O(\frac{k^2c^6}{a^6}+\frac{kc^4}{a^4}))\\
&=&h_a(u)(1+O(\frac{(\log k)^{36}}{k}))+G_a(u), 
\end{eqnarray*}
where 
$$G_a(u)=\sum_{|c|\leq (\log k)^5} e^{-\frac{kc^2}{2a^2}-cu}\frac{kc^3}{a^3}=h_a(u)O(\frac{(\log k)^{20}}{k^{1/2}}). $$
Similarly
$$g_a''(u)=h_a''(u)(1+O(\frac{(\log k)^{46}}{k}))+G_a''(u),  $$
and 
$$G_a''(u)=h_a''(u)O(\frac{(\log k)^{30}}{k^{1/2}}). $$
Notice both $G_a(u)$ and $G_a''(u)$ are odd functions, so 
$$\int_{-(\log k)^2}^{(\log k)^2} \frac{g_a''(u)du}{g_a(u)}=(1+O(\frac{(\log k)^{60}}{k}))\int_{-(\log k)^2}^{(\log k)^2} \frac{h_a''(u)du}{h_a(u)}. $$
The lemma then follows from the fact that the last integral can be replaced by the integral over $(-\infty, \infty)$, with a possibly $\epsilon(k)$ error. The proof is similar to the previous arguments, and we omit it here. 
\end{proof}

Now the following elementary lemma is crucial for our purpose. 

\begin{lem}\label{key}
For all $a>0$,  we have $$\int_{-\infty}^{\infty}\frac{h_a''(u)du}{(h_a(u))^2}=2$$
\end{lem}
\begin{proof}
For simplicity let $b=\frac{k}{2a^2}$, and by abuse of notation we will $h_b(u)=\sum_{c\in \Z} e^{-bc^2}e^{-cu}$.  Notice $h_b(x)=H_b(x)e^{-u^2/(4b)}$, where $H_b(x)=\sum_{c\in \Z}e^{-b(c-\frac{u}{2b})^2}$. Since the summation is for all integers, we see that $H_b(u)$ is periodic with period $2b$. So 
$$\int_{-\infty}^{\infty}\frac{dx}{h_b(u)}=\int_{0}^{2b}\frac{\sum_{c\in \Z}e^{-\frac{(u+2bc)^2}{4b}}}{H_b(u)}du=\int_0^{2b}du=2b.$$
It is easy to justify that differentiating with respect to $b$ commute with the integral, and notice that $h_b(u)$ satisfies the heat equation $\frac{d}{db}h_b(u)=-h_b''(u)$, we obtain 
$$\int_{-\infty}^{\infty} \frac{h_b''(u)du}{h_b(u)^2}=2. $$

\end{proof}

To sum up, the above discussion yields 
\begin{theorem}
For $a\in [k^{1/2}(\log k)^{-1}, k^{1/2}\log k]$, we have $$\mu_a=1+O(k^{-1}(\log k)^{60})$$
\end{theorem}

From the proof it is easy to see that there is a fixed $C>0$ such that the same estimate holds for $a\in [C^{-1}k^{1/2}(\log k)^{-1},  Ck^{1/2}\log k]$.
\

The above discussion of Case III suggests that the behavior of the Bergman kernel on the neck is modeled by the infinite cylinder $\C^*$. Indeed, for any $a\in [k^{1/2}(\log k)^{-1}, k^{1/2}\log k]$ we know the section $z^a$ is concentrated in an annuli neighborhood of the circle $\log\frac{1}{|z|^2}=\frac{k-2}{a}$. If we change to cylindrical coordinates $z=e^{-(\xi/2+(k-2)/a)}$, where $\xi=u+it$. Then we see the measure $|z|^{2a}e^{-k\Phi_0}\omega_{0}=|z|^{2a-2} (\log \frac{1}{|z|^2})^{k-2}dzd\bar z$ is to the leading order term approximated by the measure 
$d\mu_0=e^{-\frac{a^2u^2}{2k}}dudt$ on the cylinder. On $\C^*$, we can define a $L^2$ norm on the space of all holomorphic functions using the measure $d\mu_0$. It is easy to see 
$$\|z^c\|^2=e^{\frac{kc^2}{2a^2}}. $$
The corresponding Bergman kernel 
$$\rho(\xi)=\sum_{c\in \Z} e^{-\frac{a^2u^2}{2k}-\frac{kc^2}{2a^2}+cu}=\sum_{c\in \Z} e^{-\frac{a^2}{2k}(u-\frac{kc}{a^2})^2}.$$
Notice we can also understand this as the Bergman kernel on $\C^*$ (up to constant multiple), endowed with the hermitian metric $e^{-\frac{a^2u^2}{2k}}$ whose curvature form is the flat cylindrical metric. Geometrically, on this neck the hyperbolic metric is approximated by a flat cylinder, and
our discussion above makes precise that this model approximates $\rho_{0, k}$ when $k$ is large. 

\section{General case}
We use the same setup of the introduction. Let $(X, D)$ be a log Riemann surface, and $L$ be an ample line bundle over $X$ endowed with the (singular) hermitian metric $h$ whose curvature form is the hyperbolic metric $\omega$. Let $\Phi_k: X\rightarrow \C\P^{N_k}$ be the map defined in the introduction. Our goal is to estimate the asymptotics of $\|\mu(\Phi_k(X), \Phi_k(D), \frac{2}{3})\|_2$ as $k\rightarrow\infty$, using a particular choice of orthonormal basis of $\mathcal H_k$. 

First given any orthonormal basis $\{s_\alpha\}$ of $\mathcal H_k$, we re-write (\ref{eqn1-2}) as 
$$\frac{3}{2}\mu(X, D, \frac{2}{3})=\mu_X+\frac{1}{2}\mu_D-\tilde c_k I, $$
where 
$$\mu_X=\int_X \langle s_\alpha, s_\beta\rangle_h \rho_k^{-1}\omega_k; $$
$$\mu_D=\sum_{i=1}^d \rho_k(p_i)^{-1} \langle s_\alpha(p_i), s_\beta(p_i)\rangle_h. $$
Using Riemann-Roch formula, we obtain 
$$\tilde c_k=\frac{kl-d+\frac{d}{2}}{kl-d-g+1}=1-\frac{S}{2}k^{-1}+O(k^{-2}), $$
where $S=-\frac{d+2g-2}{l}$ is the scalar curvature of $\omega$, by our normalization. 

Now let $D=\{p_1, \cdots, p_d\}$. For each $i$ we can find a local holomorphic coordinate chart $(U_i, z)$ of $X$ centered at $p_i$, such that $\omega=-\frac{2}{S}\omega_0$ on $U_i$, and a local holomorphic section $e_i$ of $L$ over $U_i$, with $|e_i|^2=e^{\frac{2}{S}\Phi_0}$, where $\omega_0$ and $\Phi_0$ are defined in the beginning of Section 2. We may assume $U_i=\{|z|<R\}$ for some $R<1$ and $U_i\cap U_j=\emptyset$ if $i\neq j$. Inside each  $U_i$ we are essentially reduced to the model case studied in Section 2, with a possible change of $k$ by $-\frac{2}{S}k$ (notice in the whole discussion there $k$ does not have to be an integer). For the calculation below to make the notation simpler we will without loss of generality assume $S=-2$.

Fix a smooth cut-off function $\chi_i$ that equals $1$ in $U_i$, and vanishes outside a small neighborhood of $U_i$.  To obtain global sections of $L^k$, we use H\"ormander's $L^2$ estimate. The following lemma is well-known, see for example \cite{Tian1990On}. 

 \begin{lem}
 Suppose $(M,g)$ is a complete \kahler manifold of complex dimension $n$, $\mathcal L$ is a line bundle on $M$ with hermitian metric $h$. If 
 $$\langle-2\pi i \Theta_h+Ric(g),v\wedge \bar{v}\rangle_g\geq C|v|^2_g$$
 for any tangent vector $v$ of type $(1,0)$ at any point of $M$, where $C>0$ is a constant and $\Theta_h$ is the curvature form of $h$. Then for any smooth $\mathcal L$-valued $(0,1)$-form $\alpha$ on $M$ with $\bar{\partial}\alpha=0$ and $\int_M|\alpha|^2dV_g$ finite, there exists a smooth $\mathcal L$-valued function $\beta$ on $M$ such that $\bar{\partial}\beta=\alpha$ and $$\int_M |\beta|^2dV_g\leq \frac{1}{C}|\alpha|^2dV_g$$
 where $dV_g$ is the volume form of $g$ and the norms are induced by $h$ and $g$.
 \end{lem}

Fix $k$ large so that the assumption of the Lemma is satisfied in our setting with $M=X\setminus D$ and $\mathcal L=L^k$. For a positive integer $a\leq k^{3/4}$, we apply the Lemma to $\alpha_{i, a}=\bp(\chi_i \tau_a z^a e_i^{\otimes k})$ (where $\tau_a=(\frac{ a^{k-1}}{2\pi(k-2)!})^{1/2}$ is the normalization constant appearing in Lemma 2.1) on $X\setminus D$ and obtain the corresponding $\beta_{i, a}$. Then the section  $s_{i, a}:=\chi_i \tau_a z^a e_i^{\otimes k}-\beta_{i, a}$ is holomorphic over $X\setminus D$, and the $L^2$ integrability condition guarantees that $s_{i, a}$ extends to a section in $\mathcal H_k$.

By our discussion in Section 2, we know $\bp \chi_i$ is supported in the region where $|\tau_a z^a_i e_i^{\otimes k}|_h$ is $\epsilon(k)$ for all $a\leq k^{3/4}$,  and also $\|\chi_i \tau_a z^a e_i^{\otimes k}\|=1+\epsilon(k)$, where we denote by $\|\cdot\|$ the global $L^2$ norm measured with respect to the obvious metrics. So by the estimate in the above Lemma we get $\|s_{i, a}\|^2=1+\epsilon(k)$. Similarly for $1\leq a, b\leq k^{3/4}$ we have 
$$\langle s_{i, a}, s_{j, b}\rangle=\delta_{ab}\delta_{ij}+\epsilon(k), $$
where $\langle\cdot, \cdot\rangle$ denotes the obvious global $L^2$ inner product. 
We should remind the reader that the estimates for the error is of the size $\epsilon(k)$, which is independent of the indices, so  when adding them up we still have the size $\epsilon(k)$.

We may assume  $\{s_{i, a}|i=1, \cdots, d; a\leq k^{3/4}\}$ is orthonormal by possibly applying a linear transformation of the form $I+A$ where, by the remark for Lemma \ref{integral} the entries of $A=(a_{ij})$ satisfy $\sup|a_{ij}|=\epsilon(k)$. Now we let $\{s_{\tilde \gamma}\}$ be an arbitrary orthonormal basis of the orthogonal complement of the span of $\{s_{i, a}|i=1, \cdots, d; a\leq k^{3/4}\}$ in $\mathcal H_k$.

Now one can prove Theorem \ref{thm1-1} using the above chosen basis, based on the arguments of Section 2 and the known asymptotic expansion of Bergman kernel away from the punctures. For readers' convenience we include the details here, but we should point out that the discussion below is essentially straightforward. 

For each $i$, we denote by $V_i$ the subset of $U_i$ consisting of points with $(\log \frac{1}{|z|^2})^{-1}\leq k^{-1/2}\log k$, and by $W_i$ the subset of $U_i$ consisting of points with $(\log \frac{1}{|z|^2})^{-1}\leq k^{-3/8}$. As in Section 2, points in $V_i$ have injectivity radius smaller than $\pi k^{-1/2}\log k$.  
For a point $x$ outside $\bigcup_{i=1}^d V_i$, the usual proof of the Bergman kernel asymptotics goes through, and yield a uniform expansion (in the $C^2$ sense)\footnote{Indeed one can show the error term is $\epsilon(k)$ since $\omega$ has constant curvature, see for example \cite{Donaldson15}.}
\begin{equation} \label{outside}\rho_k(x)=\frac{1}{2\pi}(k-1+O(k^{-1})), 
\end{equation}
and so 
\begin{equation} \label{outside2}
\omega_k=\rho_k \omega(1+k^{-1}+O(k^{-2})). \end{equation}

\begin{lem}\label{sgamma}
	We have the following estimate:
		$$\sup_{\tilde{\gamma}}\sup_i \sup_{z\in W_i}\frac{|s_{\tilde{\gamma}}(z)|^2}{\rho_{k,0}(z)}=\epsilon(k).$$
\end{lem}

\begin{proof}
Fix $i$, within $U_i$ we can write $s_{\tilde\gamma}=f_{\tilde\gamma}e_i^{\otimes k}$ for a holomorphic function $f_{\tilde\gamma}$. Let $f_{\tilde{\gamma}}=\sum c_a\tau_az^a$ be the Taylor expansion around $p_i$. We are interested in the estimates of various quantities for large $k$, so the estimates below will always be understood for $k$ sufficiently large, but are independent of $\tilde\gamma$ and $i$.

\

\textbf{Claim 1.} For all $a\leq k^{3/4}$ we have $|c_a|=\epsilon(k)$. 

To prove this, we use the fact that $s_{\tilde{\gamma}}$ is $L^2$ orthogonal to $s_{i, a}$. Since $s_{i, a}=\chi_i \tau_a z^a e_i^{\otimes k}+\beta_{i, a}$ with $\|\beta_{i, a}\|=\epsilon(k)$, and again by the discussion of Section 2 we know $\int_{U_i} |\tau_a z^ae_i^{\otimes k}|^2\omega=1+\epsilon(k)$ (since the integral is concentrated in the annulus $|t-\frac{k-2}{a}|\leq \frac{k^{1/2}\log k}{a}$, where we adopt the notation of Section 2 to denote $t=\log \frac{1}{|z|^2}$), it is then easy to obtain the conclusion. 

\

\textbf{Claim 2.} For all $a \leq -\frac{k}{2\log R}$, we have $|c_a|^2\leq 3$ for large $k$. 

This follows from similar consideration as above, using the fact that $$\int_{U_i} |\tau_az^ae_i^{\otimes k}|^2\omega\geq 1/3$$ for $a$ in this range. 

\

Now we define a function $q_{a}(z)=|\frac{c_a\tau_az^a}{\tau_bz^b}|^2$, where $b=k^{5/8}$. Denote $\gamma_a=q_a(z)$ when $\log |z|^2=-k/b$.

\

\textbf{Claim 3}. For all $a \geq -\frac{k}{2\log R}$, we have 
$\gamma_a=O(e^{-\frac{1}{3}ak^{3/8}}).$

 As we have seen in the local calculation in Section 2, for $k$ large we have
 $$\int_{\frac{k}{b}-\frac{k^{1/2}\log k}{b}\leq \log \frac{1}{|z|^2}\leq \frac{k}{b}} |\tau_b z^b e_i^{\otimes k}|^2\omega\geq 1/3. $$
 We denote the annulus $\{z|\frac{k}{b}-\frac{k^{1/2}\log k}{b}\leq \log \frac{1}{|z|^2}\leq \frac{k}{b}\}$ by $A_b$. 
 
 For $a\geq k^{3/4}$, when $k$ is large we have $a>b$ hence $q_a$ is increasing in $|z|$, hence we have 
 $$\int_{A_b} |c_a\tau_a z^a e_i^{\otimes k}|^2\omega=\int_{A_b} q_a(z)|\tau_b z^b e_i^{\otimes k}|^2\omega\geq \frac{1}{3}\gamma_a.$$
 On the other hand, we have 
 $$\int_{A_b} |z^ae_i^{\otimes k}|^2\omega=2\pi \int_{\frac{k}{b}-\frac{k^{1/2}\log k}{b}}^{\frac{k}{b}} e^{(k-2)\log t-at}dt\leq e^{-\frac{1}{2}ak^{3/8}}.$$
Therefore,
$$\gamma_a\leq |c_a\tau_a|^2e^{-ak^{3/8}/2}.$$
Now since $\|s_{\tilde \gamma}\|=1$, we have 
$$\int_{|z|\leq R} |c_a\tau_az^ae_i^{\otimes k}|^2\omega\leq 1.$$
Notice
$$\int_{|z|\leq R} |z^a e_i^{\otimes k}|^2\omega=2\pi \int_{-2\log R}^\infty e^{(k-2)\log t-at}dt\geq e^{-a(1-2\log R)+k\log (1-2\log R)},$$
so we obtain 
$$|c_a\tau_a|^2\leq e^{a(1-2\log R)-k\log (1-2\log R)}, $$
Hence 
$$\gamma_a\leq e^{-\frac{1}{2}ak^{3/8}+a(1-2\log R)-k\log (1-2\log R)}\leq e^{-\frac{1}{3}ak^{3/8}}, $$
and \textbf{Claim 3} is proved. 


\

Therefore, for $z\in W_i$, we have by \textbf{Claim 1} that 
\begin{equation}\label{eqn3-1}
\frac{|\sum_{a\leq k^{3/4}} c_a\tau_az^a|^2}{\sum_a |\tau_a z^a|^2}=\epsilon(k), 
\end{equation}
and by \textbf{Claim 3}
\begin{equation}\label{eqn3-2}
\frac{|\sum_{a\geq -\frac{k}{2\log R}} c_a\tau_az^a|^2}{\sum_a |\tau_a z^a|^2}\leq( \sum_{a\geq -\frac{k}{2\log R}}e^{-ak^{3/8}/6})^2=\epsilon(k). 
\end{equation}

For $a\in [k^{3/4}, -\frac{k}{2\log R}]$, we first notice

\

\textbf{Claim 4}. For all $a\geq k^{3/4}$ and all $z\in W_i$ we have 
$$\frac{|\tau_a z^a|^2}{\sum_{r\geq 1}|\tau_r z^r|^2}=\epsilon(k). $$ 
The proof of this follows from similar discussion as in Section 2, the main point being that for $z\in W_i$, the main contribution to the sum in the denominator comes from terms where $r$ is around $(\log \frac{1}{|z|^2})^{-1}k\leq k^{5/8}$, which is much smaller than $k^{3/4}$ when $k$ is large. To be more precise, we can write $|\tau_a z^a|^2=a^{k-1}|z|^{2a}=e^{(k-1)\log a+a\log |z|^2}$. Let $P(y)=(k-1)\log y+\log |z|^2 y$, it is easy to see that when $z\in W_i$ and $y\geq k^{3/4}-1$, $P$ is concave and decreasing in $y$, hence we have 
$$P(a)-P(a-1)\leq P'(a-1)\leq -k^{1/2}.$$
From this \textbf{Claim 4} follows easily.

\

By \textbf{Claim 2} and \textbf{Claim 4} we also have 
\begin{equation}\label{eqn3-3}
\frac{|\sum_{k^{3/4}\leq a\leq -\frac{k}{2\log R}} c_a\tau_az^a|^2}{\sum_a |\tau_a z^a|^2}=\epsilon(k). 
\end{equation}

The conclusion of the Lemma then follows from the combination of (\ref{eqn3-1}), (\ref{eqn3-2}) and (\ref{eqn3-3}). 
%
\end{proof}
As a consequence,  we obtain the following lemma, which essentially shows the ``localilty" of Bergman kernel in a neighborhood of the hyperbolic cusp.
 
\begin{lem} \label{Lem3-3}
We have 
\begin{equation}\label{rhok}
\sup_{i}\sup_{z\in W_i} |\frac{\rho_k(z)}{\rho_{k, 0}(z)}-1|=\epsilon(k).
\end{equation}
\end{lem}
\begin{proof}
	Given Lemma \ref{sgamma}, we only need to show for all $a\leq k^{3/4}$ and $z\in W_i$, 
	\begin{equation} \label{eqn3-5}
	\rho_{k, 0}(z)^{-1}|\beta_{i, a}(z)|^2=\epsilon(k). 
	\end{equation}
In the above local coordinate we can write $\beta_{i, a}=\sum_a c_a\tau_az^ae_i^{\otimes k}$. Notice we have $\|\beta_{i, a}\|_{L^2}=\epsilon(k)$, hence $\langle\beta_{i, a}, s_{i, a}\rangle=\epsilon(k)$, so one can see that the same estimates for $c_a\tau_a$ in the proof of Lemma 3.2 also holds here, and we obtain the conclusion. 	
\end{proof}
\begin{cor}
We have
\begin{equation}\label{eqn3-8}\sup_{\tilde{\gamma}}\sup_i[\int_{V_i}\frac{|s_{\tilde{\gamma}}|^2}{\rho_k}\omega_k+\int_{V_i} |s_{\tilde\gamma}|^2\omega]=\epsilon(k)
\end{equation}
\end{cor}

\begin{proof}
This is straightforward, just by noticing that $\int_{V_i}\omega_k\leq \int_X\omega_k=O(k)$, and $\int_{V_i}\rho_k\omega\leq \int_X\rho_k\omega=\dim H^0(X, L^k)=O(k)$.
\end{proof}

Using (\ref{outside}), (\ref{outside2}) and (\ref{eqn3-8}), we get
\begin{eqnarray} \label{eqn3-9}
\mu_X(\tilde\gamma, \tilde\gamma)&=&\int_{X\setminus \bigcup_i V_i} |s_{\tilde\gamma}|^2 (1+k^{-1}+O(k^{-2}))\omega+\int_{\bigcup_i V_i}|s_{\tilde\gamma}|^2\rho_k^{-1}\omega_k\\
&=& 1+k^{-1}+O(k^{-2}). \nonumber
\end{eqnarray}

\

Now for $\alpha=(i, a)$ with $a\geq 2k^{1/2}\log k$, we claim 

\begin{equation}
\sup_{z\in V_i} \frac{|\tau_a z^a|^2}{\rho_{k,0}(z)}=\epsilon(k). 
\end{equation}

Indeed, this follows from similar argument as in the proof of \textbf{Claim 4} in Lemma \ref{sgamma}. The point is that the function $P(y)$ is also concave and decreasing when $z\in V_i$ and $y\geq 2k^{1/2}\log k$, and we have $P'(y-1)\leq -\frac{1}{2}k^{1/2} (\log k)^{-1}$. So together with (\ref{rhok}) and (\ref{eqn3-5}) we also have 

$$\sup_{i}\sup_{ a\in [2k^{1/2}\log k, k^{3/4}]}(\int_{V_i} \frac{|s_{i, a}|^2}{\rho_{k, 0}}\omega_k+\int_{V_i}|s_{i, a}|^2\omega)=\epsilon(k). $$
Hence in this case we also obtain
\begin{equation} \label{eqn3-11}
\mu_X(\alpha, \alpha)= 1+k^{-1}+O(k^{-2}). 
\end{equation}

\
 
\begin{lem}
In $W_i$, we have
	\begin{equation}
		\omega_k-\omega_{k,0}=\epsilon(k)\omega, 
	\end{equation}
	where $\omega$ is the hyperbolic metric on $X$. 
\end{lem}
\begin{proof}	
	
	Write $\rho_{k, 0}^{-1}\rho_k=1+F_k$, then as shown in the proof  of Lemma \ref{Lem3-3}, we have the pointwise estimate $F_k=\epsilon(k)$ and $F_k$ can be written as the sum of contributions from $s_{\tilde\gamma}$ and $\beta_{i, a}$. In local coordinates, both types of error terms are of the form $E(z)=\frac{|\sum_{a\geq 1} c_a\tau_a z^a|^2}{\sum_{a\geq 1}|\tau_az^a|^2}$, where as shown above, the $c_a$ satisfies the estimates in the proof of Lemma \ref{sgamma}. 
	
	Notice $\omega_k-\omega_{k,0}=-i\p\bp\log (1+F_k)$, so it suffices to estimate $\p\bp F_k$ and $\p F_k\wedge \bp F_k$ using $\omega$. Then it is further reduced to show the following
$$\p\bp E=\epsilon(k)\omega, \p E\wedge \bp E =\epsilon(k)\omega. $$
These can be proved in the same way as in the proof of Lemma \ref{sgamma}. 
%
%
%
For simplicity, we denote $f=\sum_{a\geq 1} c_a\tau_a z^{a-1}$, and $g=\sum_{a\geq 1}\tau_a z^{a-1}$. Then $E=\frac{f}{g}$, and
\begin{eqnarray*}
	\partial \bar{\partial}E&=&\frac{\partial \bar{\partial}f}{g}-\frac{\partial g\wedge \bar{\partial}f+\partial f\wedge \bar{\partial}g+f\partial  \bar{\partial}g}{g^2}+\frac{2f\partial g\wedge \bar{\partial}g}{g^3}\\
			\partial E\wedge\bar{\partial}E&=&\frac{\partial f\wedge \bar{\partial}f}{g^2}-\frac{f\partial f\wedge \bar{\partial}g+f\partial g\wedge \bar{\partial}f}{g^3}+\frac{f^2\partial g\wedge \bar{\partial}g}{g^4}
\end{eqnarray*}

One can show every single term in the above is indeed $\epsilon(k)\omega$. 
For example we will treat the first term in $\p\bp E$. Notice $\p\bp f=|\sum_{a\geq 2}(a-1)c_{a}\tau_{a}z^{a-2}|^2dzd\bar z$, and $\omega=\frac{1}{|z|^2(\log |z|^2)^2}dzd\bar z$. So if $\log |z|^2\geq-k$, we have $(\log|z|^2)^2=O(k^2)$, so we obtain as in the proof of Lemma \ref{sgamma} that 
$$\frac{\p\bp f}{g\omega}=\epsilon(k). $$ When $\log |z|^2\leq -k$, by the discussion of Section 2 we know $\tau_1^2\leq g\leq 2\tau_1^2$, and 
$$|\sum_{a\geq 2} (a-1)c_a\tau_a z^{a-2}|^2\leq 2\tau_2^2\epsilon_k+2|\sum_{a\geq 3}(a-1)c_a\tau_az^{a-2}|^2.$$
Notice $|\tau_2|^2/|\tau_1|^2=2^{k-1}$, and using the convexity similar to the proof of \textbf{Claim 4} in Lemma \ref{sgamma}, we can see for $a\geq 3$, $|\tau_az^{a-2}|^2/|\tau_1|^2\leq e^{-(a-1)}$. Since $|z|^2\leq e^{-k}$, we get 
$$\frac{\p\bp f}{g\omega}=\epsilon(k). $$
The estimates for the other terms follow similarly. 
\end{proof}

\begin{rmk}
With more work, it is possible to get higher order derivative estimates for the error term $\rho_{k}-\rho_{k, 0}$, but these are not needed for our current purpose in this paper. 
\end{rmk}

\begin{lem}
	For $\alpha=(i, a)$ with $a\leq 2 k^{1/2}\log k$, we have
\begin{equation}
\mu_X(\alpha, \alpha)=\int_{W_i} \frac{|\tau_{ a}z^a|^{2}}{\rho_{k, 0}}\omega_{k, 0}+\epsilon(k)
\end{equation}
\end{lem}
\begin{proof}
We write 
$$\mu_X(\alpha, \alpha)=\int_{W_i} \frac{|s_{i, a}|^2}{\rho_{k}}\omega_k+\int_{X\setminus W_i}\frac{|s_{i, a}|^2}{\rho_k}\omega_k. $$
Since outside $W_i$ we have the expansion (\ref{outside}) and (\ref{outside2}), we get 
$$\int_{X\setminus W_i} \frac{|s_{i, a}|^2}{\rho_k}\omega_k=\int_{X\setminus W_i} |s_{i, a}|^2(1-\frac{S}{2}k^{-1}+O(k^{-2}))\omega. $$
Notice 
$$\int_{X\setminus W_i} |s_{i, a}|^2\omega\leq \int_{X}|\beta_{i, a}|^2\omega+\int_{U_i\setminus W_i} |\tau_az^a e_i^{\otimes k}|^2\omega=\epsilon(k)+\int_{U_i\setminus W_i} \frac{|\tau_a z^a|^2}{\sum_{d}|\tau_d z^d|^2}\rho_{k, 0}\omega.  $$
Now as in the proof of \textbf{Claim 4} in Lemma \ref{sgamma} one can easily see that $\frac{|\tau_a z^a|^2}{\sum_{d}|\tau_d z^d|^2}=\epsilon(k)$ for $a\leq 2k^{1/2}\log k$. On the other hand, outside $W_i$ we have the usual expansion of $\rho_{k, 0}$  as (\ref{usualexpansion}), hence $\rho_{k, 0}=O(k)$ on $U_i\setminus W_i$. So we have 
$$\int_{X\setminus W_i} |s_{i, a}|^2\omega=\epsilon(k). $$

Now by the above discussion we get
$$\int_{W_i}\frac{|s_{i, a}|^2}{\rho_k}\omega_k=\int_{W_i}|\tau_az^ae_i^{\otimes k}|^2(1+\epsilon(k))\rho_{k, 0}^{-1}\omega_{k,0}+\int_{W_i}\frac{|\beta_{i, a}|^2}{\rho_{k, 0}}(1+\epsilon(k))\omega. $$
The conclusion follows since the last term is $\epsilon(k)$ by similar arguments as in the proof of Lemma \ref{Lem3-3}.
\end{proof}

By the results of Section 2, we obtain
\begin{equation}\label{eqn3-14}
	\mu_X(\alpha, \alpha)=
	\left\{
	\begin{array}{ll}
	\frac{1}{2}+\epsilon(k),   & a=1;  \\
	 1+O(k^{-1}(\log k)^{60}) & a>1.
	\end{array}
	\right.
	\end{equation}

\

Similarly, for the off-diagonal term of $\mu_X$, we have 

\begin{lem}For $\alpha\neq \beta$, 
	\begin{equation}\label{eqn3-15}
	\mu_X(\alpha, \beta)=
	\left\{
	\begin{array}{ll}
	\epsilon(k),   & \alpha=(i, a), \beta=(j, b);  \\
	O(k^{-2}) & \text{otherwise}. 
	
	\end{array}
	\right.
	\end{equation}
\end{lem}
\begin{proof}
We need to check each case separately. If $\alpha=\tilde\gamma$, then for any other unit norm section $s$ which is $L^2$ orthogonal to $s_{\alpha}$, we can write 
$$\int_{X}\langle s_\alpha, s\rangle \rho_k^{-1}\omega_k=\int_{X\setminus V_i} \langle s_\alpha, s\rangle \rho_k^{-1}\omega_k+\int_{V_i}\langle s_\alpha, s\rangle \rho_k^{-1}\omega_k. $$
For the first term we have 
\begin{eqnarray*}
	\int_{X\setminus V_i} \langle s_\alpha, s\rangle \rho_k^{-1}\omega_k&=&\int_{X\setminus V_i} \langle s_\alpha, s\rangle (1-\frac{S}{2}k^{-1}+O(k^{-2}))\omega\\
	&=&-\int_{V_i}\langle s_\alpha, s\rangle (1-\frac{S}{2}k^{-1})\omega+O(k^{-2}).
\end{eqnarray*}
Using (\ref{eqn3-8}) we see this is $\epsilon(k)+O(k^{-2})=O(k^{-2})$. For the second term we similarly apply previous arguments to see it is $\epsilon(k)$. 

Now if $\alpha=(i, a)$ and $\beta=(i, b)$ with $a\neq b$, then
$$\int_{X}\langle s_\alpha, s_\beta\rangle \rho_k^{-1}\omega_k=\int_{X\setminus U_i}\langle s_\alpha, s_\beta\rangle (1-\frac{S}{2}k^{-1}+O(k^{-2}))\omega+\int_{U_i}\langle s_\alpha, s_\beta\rangle \rho_k^{-1}\omega_k. $$
Since $s_{i, a}=\chi_i \tau_az^a e_i^{\otimes k}+\beta_{i, a}$ with $\|\beta_{i, a}\|=\epsilon(k)$, we easily see the first term is $\epsilon(k)$. The second term is also $\epsilon(k)$ because $a\neq b$, $\rho_k^{-1}\omega_k=\rho_{k, 0}^{-1}(\omega_{k,0}+\epsilon(k)\omega)$,  $|\beta_{i, a}|^2\rho_{k, 0}^{-1}=\epsilon(k)$ in $W_i$, and $\int_{U_i\setminus W_i}|\beta_{i, a}|^2\omega=\epsilon(k)$. 

For $\alpha=(i, a)$ and $\beta=(j, b)$ with $i\neq j$ the conclusion follows similarly.

\end{proof}

Finally we have
\begin{equation}\label{eqn3-16}
\mu_D(\alpha, \alpha)=
\left\{
                                                               \begin{array}{ll}
1+\epsilon(k),   & \alpha=(i, 1);\\
\epsilon(k), & \text{otherwise.} \\

                                                               \end{array}
                                                             \right.
\end{equation}

Putting together (\ref{eqn3-9}), (\ref{eqn3-11}), (\ref{eqn3-14}), (\ref{eqn3-15}), (\ref{eqn3-16}), we get 

\begin{eqnarray*}
\|\mu_X+\frac{1}{2}\mu_D-\tilde c_k \|_2^2&=&O(k^2\cdot k^{-4}+O(k^{-2}(\log k)^{120}\cdot k^{1/2}\log k)\\&=&O(k^{-3/2}(\log k)^{121})
\end{eqnarray*}

This finishes the proof of Theorem \ref{thm1-1}.

\section{Explicit study of Chow stability}

We assume $X=\PP^1$ and $D$ the union of $d$ distinct points $p_1, \cdots, p_d$. Consider an embedding of $(X, D)$ into $\PP^N$ using $H^0(X, L^k)$ where $L=[D]$ and $N=kd$.

\begin{theorem}

Suppose $d\geq 2$, then $(X, D)$ is $\lambda$-semistable if and only if $\lambda\in [\lambda_k, 1]$, and $\lambda$-stable if and only if $\lambda\in (\lambda_k, 1]$. Here $\lambda_k=2/(d+1)$ when $k=1$ and $\lambda_k=\frac{2kd+2}{3kd+d+1}$ when $k\geq 2$. 
\end{theorem}

\begin{proof}
Clearly $(X, D)$ is always $1$-balanced. A simple fact is that (c.f. \cite{ssun}) the set of $\lambda$ for which $(X, D)$ is $\lambda$-semistable form an interval of the form $[\lambda_k, 1]$ for some $\lambda_k\in (0, 1)$. The point is to determine $\lambda_k$. The pair $(X, D)$ is not $\lambda_k$-balanced but there is another pair $(X_0, D_0)$ in the closure of the $SL(N+1;\C)$ orbit of $(X, D)$ (in an appropriate Chow variety) which is $\lambda_k$-balanced. When $d=2$ this is proved in \cite{ssun} where $X_0$ is constructed as a chain of linear rational curves in $\PP^N$. Now we focus on the case $d\geq 3$. We will construct these by induction. When $k=1$, we let $q_i$ be the $i$-th coordinate point of $\PP^d$ for $i=1, \cdots, d+1$. Then we let $D_0=\{q_1, \cdots, q_d\}$, and $X_0$ be the union of all lines connecting $q_i$ with $q_{d+1}$.  A straightforward calculation shows that $(X_0, D_0)$ is $\lambda_1$-balanced, and it is also easy to see that $(X_0, D_0)$ is in the $SL(N+1;\C)$ orbit of $(X, D)$. This shows that $(X, D)$ is strictly $\lambda_1$-semistable and hence we are done with $k=1$. Now suppose the conclusion holds for $k=m-1$ and we consider the case $k=m$.  Then again we denote by $q_i$ the $i$-th coordinate point of $\PP^N$ for $i=1, \cdots, N+1$. Let $X_0$ be the union of a smooth rational normal curve $Y$ in $\PP^{N-d}$ (viewed naturally as the a subspace of $\PP^N$ which contains $q_1, \cdots, q_{N-d+1}$) which passes through the co-ordinate points, and the lines connecting the $q_i$ with $q_{N-d+1+i}$. Let $D_0=\{q_{N-d+1+j}|j=1, \cdots, d\}$, and $E=\{q_1, \cdots, q_d\}$. This is in the closure of the $SL(N+1;\C)$ orbit of $(X, D)$ (this can be alternatively seen as a deformation to the normal cone). Now notice since $d\geq 3$ we have $\lambda_{m-1}<2/3$, by the induction hypotheses we may assume $(Y, E)$ is $2/3$-balanced in $\PP^{N-d}$. Then it is again a straightforward calculation that $(X_0, D_0)$ is $\lambda_m$-balanced in $\PP^N$. More precisely, we obtain this value of $\lambda_m$ by solving the equation
$\frac{1}{2}\lambda+(1-\lambda)=\lambda \frac{2N+d}{2(N+1)}$. By the same reason as the case $k=1$ we see the conclusion holds for $k=m$.  
\end{proof}

\begin{figure}
 \begin{center}
  \includegraphics[width=0.8 \columnwidth]{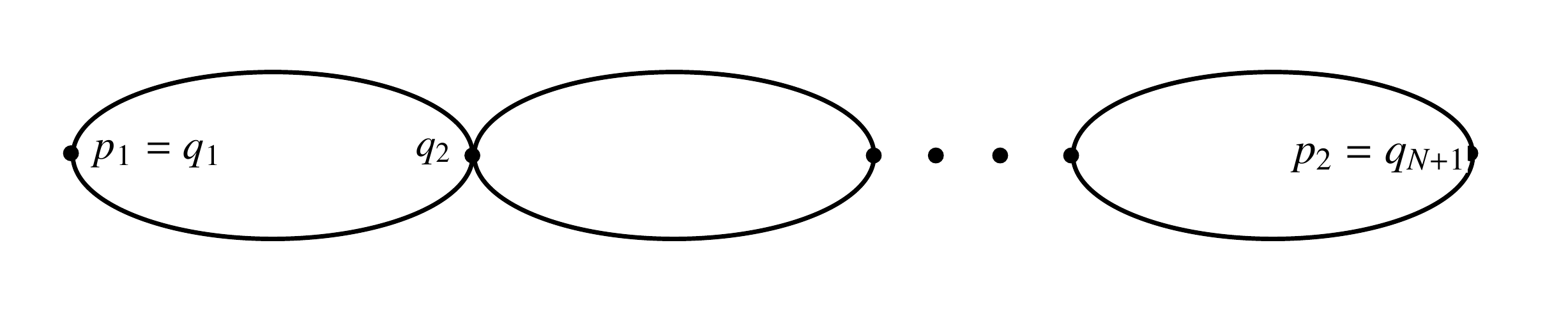}
  \caption{A $\frac{2}{3}$-balanced pair in $\P^N$ with $D=\{p_1, p_2\}$}
  \label{fig1}
 \end{center}
 \end{figure}

\begin{figure}
 \begin{center}
  \includegraphics[width=0.8 \columnwidth]{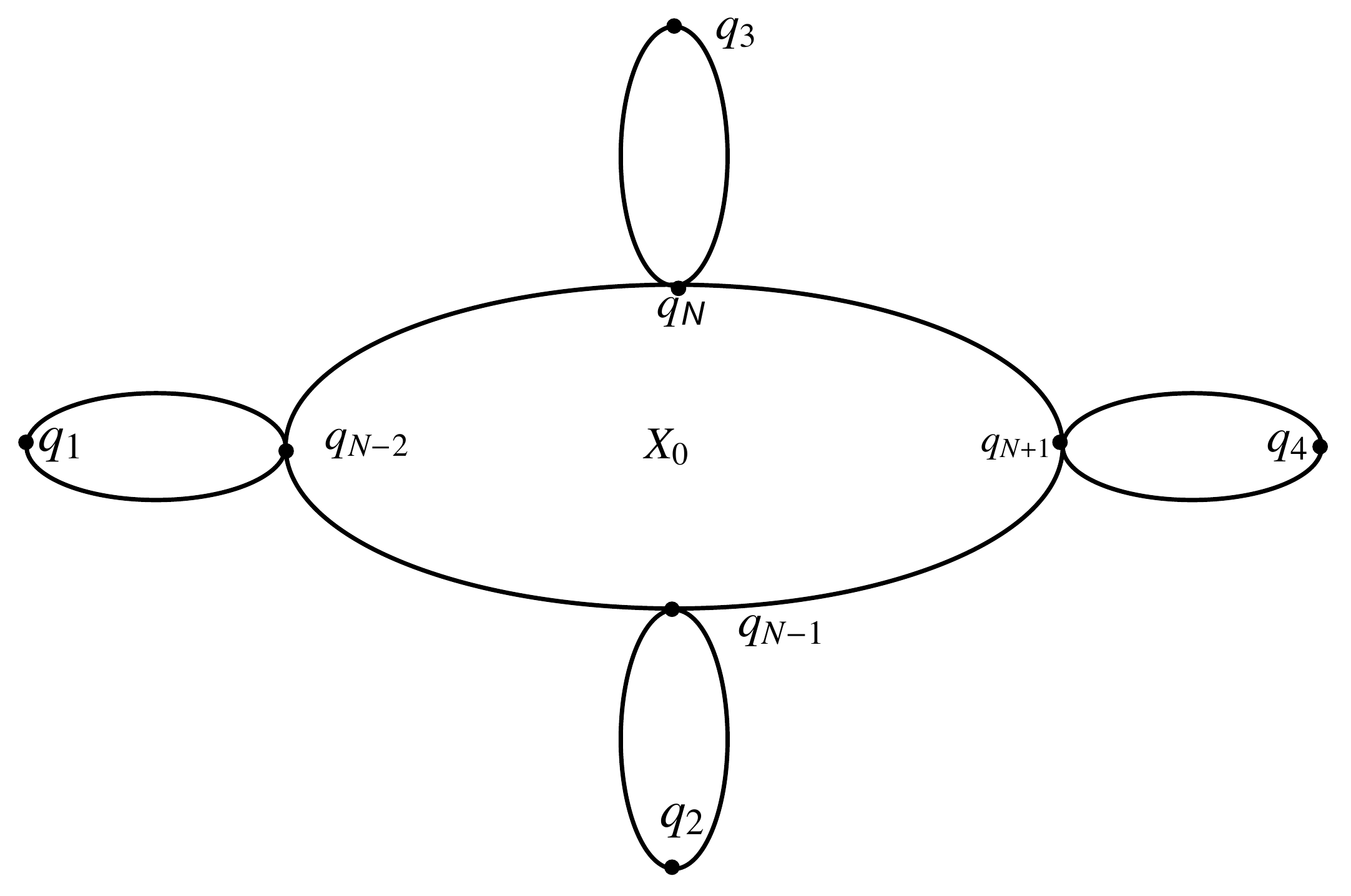}
  \caption{A $\lambda_k$-balanced pair in $\P^N$ with $d=4$ and $k\geq 2$}
  \label{fig2}
 \end{center}
 \end{figure}

\bibliographystyle{plain}

\bibliography{references}

\end{document}